\documentclass[12pt,sloppy]{article}
\usepackage{graphicx}
\usepackage{subfigure}
\usepackage{latexsym}
\usepackage{amsfonts}
\usepackage{amsthm}
\usepackage{amsmath}
\usepackage{chngpage}
\usepackage{enumitem}

\newlist{case}{enumerate}{1}
\setlist[case]{label=Case \arabic*:}
\author{{\bf Maya Mohsin Ahmed }}
\title{{\LARGE {\bf A Platonic basis of Integers }}}
\date{}

\newtheorem{prop}{Proposition}[section]

\newtheorem{exm}{Example}[section]

\newtheorem{conjecture}{Conjecture}[section]

\begin{document}     
\maketitle           

\begin{abstract}

In this article, we prove  that every integer can be written
as an integer combination  of exactly 4 tetrahedral numbers.
Moreover,  we compute  the modular  periodicity  of platonic
numbers.

\end{abstract}

\section{Introduction}\label{intro}
Figurate  numbers related  to the  five platonic  solids are
called  {\em  platonic   numbers}.  In  \cite{pollock},  Sir
Frederick Pollock conjectured that every positive integer is
the sum of  at most five, seven, nine,  thirteen, and twenty
one   tetrahedral   numbers,   octahedral  numbers,   cubes,
icosahedral numbers, and dodecahedral numbers, respectively.
See \cite{cg}  for a detailed study of  Platonic numbers. In
this  article, we discuss  integer combinations  of platonic
numbers instead of sums.  Thus we also include the operation
of subtraction.

Let   $t_n,  o_n,  c_n,   i_n$,  and  $d_n$
represent  the   tetrahedral  numbers,  octahedral  numbers,
cubes,  icosahedral numbers,  and the  dodecahedral numbers,
respectively.  The  following  identities are  discussed  in
\cite{cg} and \cite{kim}:
\begin{eqnarray} \label{formulas}
t_n =  \frac{n(n+1)(n+2)}{6} \\ o_n  = \frac{n(2n^2+1)}{3}
\nonumber   \\    c_n   =   n^3   \nonumber    \\   i_n   =
\frac{n(5n^2-5n+2)}{2}      \nonumber     \\      d_n     =
\frac{n(9n^2-9n+2)}{2} \nonumber
 \end{eqnarray}
We list a few numbers in the sequences:
\[ \begin{array}{llllllllllllllll}
t_n: & 1,4,10, 20, 35, 56, 84, 120, 165, 220, \dots \\

o_n: & 1, 6, 19, 44, 85,  146, 231, 344, 489, 670, \dots \\

c_n: & 1, 8, 27, 64, 125, 216, 343, 512, 729, 1000, \dots \\

i_n: & 1, 12, 48, 124, 255, 456, 742, 1128, 1629, 2260, \dots \\

d_n: & 1, 20, 84, 220, 455, 816, 1330, 2024, 2925, 4060, \dots
\end{array} \]

 In Section  \ref{differences}, we prove  that every integer
 can  be written  as  an integer  combination  of exactly  4
 tetrahedral numbers. Moreover,  we prove that integers that
 are divisible by four, six, forty five, and fifty four, can
 be  written   as  integer  combinations   of  exactly  four
 octahedral   numbers,  cubes,   icosahedral   numbers,  and
 dodecahedral numbers, respectively. We also conjecture that
 every  integer can  be written  as a  sum of  at  most five
 different platonic numbers. 

We say a  sequence $s_n$ is {\em periodic}  mod $d$ if there
exists an  integer $m$ such  that, $s_{n+m} \equiv  s_n$ mod
$d$ for every  $n$. The smallest such integer  $m$ is called
the {\em period} of $s_n$ mod $d$.

\begin{exm}{\em
Let $r_n$ denote the residues of the the tetrahedral numbers $t_n$  modulo $2$:

\[\begin{array}{cccccccccccccccccccccc}
t_n: & 1&4&10&20& 35& 56& 84& 120& 165& 220 & 286 & 364 & 455 & 560 & \dots \\
r_n: & 1&0&0& 0&   1& 0&  0&   0&   1&  0&  0&   0&   1& 0\dots
\end{array} \]
Observe that  the reduced  sequence $r_n$  repeats after
every  four terms.  Consequently, $t_{n+4}  \equiv  t_n$ mod
$2$. Hence $t_n$ has period $4$ mod $2$. }\end{exm}

In  Section  \ref{differences}, we  show  that the  platonic
numbers satisfy the following linear recurrence relation:
\begin{eqnarray} \label{lineareqns}
y_n= 4y_{n-1}-6y_{n-2}+4y_{n-3}-y_{n-4}.
\end{eqnarray}

It  is  well  known  that linear  recurrence  sequences  are
eventually periodic. See  \cite{engstorm} and the references
therein for  a discussion  of periodic sequences  defined by
linear recurrences. In  Section \ref{details}, we derive the
periods of the platonic numbers.

\section{Forward Differences.}\label{differences}

 In this section, we  use forward differences of tetrahedral
 numbers to  prove that every  integer can be written  as an
 integer combination  of exactly 4  tetrahedral numbers. Let
 $y_n$ denote  a sequence of  integers, then the  {\em first
   forward difference}, denoted  by $\Delta y_n$, is defined
 as
\[
\Delta y_n = y_{n+1} - y_n.
\]

$\Delta$  is called the  {\em forward  difference operator}.
The differences of the  first forward differences are called
the  {\em second  forward  differences} and  are denoted  by
$\Delta^2  y_i$. Continuing  thus, the  {\em  $n$-th forward
  difference} is defined as
\[
\Delta^n y_n =  \Delta^{n-1} y_{n+1} - \Delta^{n-1} y_n. 
\]

We provide the forward differences tables of the five platonic numbers below.

\[\begin{array}{llllllllllllll}

\begin{array}{cccccccccccccccccccccccc}
t_n & \Delta t_n  & \Delta^2t_n & \Delta^3 t_n &  \Delta^4 t_n  \\ \hline 
1 & 3 & 3 & 1 & 0\\ 4 & 6 & 4 & 1 & 0  \\ 10 & 10 & 5 & 1 & 0  \\ 20 & 15 & 6 & 1 & 0\\ 35 & 21 & 7 & 1 & 0\\ 56 & 28 & 8 & 1 & 0 \\ 84 & 36 & 9 & 1  & \vdots  \\ 
120 & 45 & 10 & \vdots  & \vdots \\ 165 & \vdots  & \vdots & \vdots  & \vdots \\  \vdots  & \vdots & \vdots  & \vdots & \vdots 
\end{array} 

& &

\begin{array}{cccccccccccccccccccccccc}
o_n & \Delta o_n  & \Delta^2o_n & \Delta^3 o_n &  \Delta^4 o_n  \\ \hline
1  & 5 & 8 & 4 & 0 \\
6  & 13 & 12 & 4 & 0 \\
19 & 25 & 16 & 4 & 0\\
44 & 41 & 20 & 4 & 0\\
85 & 61 & 24 & 4 & 0 \\
146 & 85 & 28 & 4 & 0 \\
231 & 113 & 32 & 4 & \vdots \\
344 & 145 & 36 & \vdots & \vdots \\
489 & 181 & \vdots & \vdots & \vdots \\
670 & \vdots  & \vdots & \vdots  & \vdots \\   \vdots  & \vdots & \vdots  & \vdots & \vdots 
\end{array} 

& & 
\begin{array}{cccccccccccccccccccccccc}
c_n & \Delta c_n  & \Delta^2c_n & \Delta^3 c_n &  \Delta^4 c_n  \\ \hline
1  & 7 & 12 & 6 & 0 \\
8  & 19 & 18 & 6 & 0 \\
27 & 37 & 24 & 6 & 0 \\
64 & 61 & 30 & 6 & 0 \\
125 & 91 & 36 & 6 & 0  \\
216 & 127 & 42 & 6 & 0  \\
343 & 169 & 48 & 6 & 0 \\
512 & 217 & 52 & 6 & \vdots \\
729  & 271 & 58 & \vdots & \vdots  \\
1000 &  \vdots  & \vdots & \vdots  & \vdots \\   \vdots  & \vdots & \vdots  & \vdots & \vdots 
\end{array} 
\\

\end{array}
\]

\[\begin{array}{llllllllllllll}

\begin{array}{cccccccccccccccccccccccc}
i_n & \Delta i_n  & \Delta^2i_n & \Delta^3 i_n &  \Delta^4 i_n  \\ \hline 
1 & 11 & 25 & 15 & 0 \\
12 & 36 & 40 & 15 & 0 \\
48 & 76 & 55 & 15 & 0 \\
124 & 131 & 70 & 15 & 0 \\
255 & 201 & 85 & 15 & 0 \\
456 & 286 & 100 & 15 & 0 \\
742 & 386 & 115 & 15 & \vdots \\
1128 & 501 & 130 & \vdots & \vdots   \\
1629 & 631 & \vdots &  \vdots & \vdots  \\
2260 &  \vdots  & \vdots & \vdots  & \vdots \\   \vdots  & \vdots & \vdots  & \vdots & \vdots 
\end{array} 

& & 

\begin{array}{cccccccccccccccccccccccc}
d_n & \Delta d_n  & \Delta^2d_n & \Delta^3 d_n &  \Delta^4 d_n  \\ \hline 
1 & 19 & 45 & 27 & 0 \\
20 & 64 & 72 & 27 & 0 \\
84 & 136 & 99 & 27 & 0 \\
220 & 235 & 126 & 27 & 0 \\
455 & 361 & 153 & 27 & 0 \\
816 & 514 & 180 & 27 & 0  \\
1330 & 694 & 207 & 27 & \vdots \\
2024 & 901 & 234 & \vdots & \vdots  \\
2925 & 1135    &  \vdots  & \vdots & \vdots     \\
4060 &  \vdots  & \vdots & \vdots  & \vdots \\   \vdots  & \vdots & \vdots  & \vdots & \vdots 
\end{array} 
\end{array}
\]

We see the fourth forward differences of all the platonic numbers are
zero.  The following identities are derived from the first forward
differences of the Platonic numbers.

\begin{eqnarray} \label{firsteqns}
\Delta t_n = t_{n+1}- t_n &=&  \frac{1}{2}n^2+\frac{3}{2}n+1.  \\ \nonumber
\Delta o_n = o_{n+1}- o_n &=& 2n^2+2n+1.  \\ \nonumber
\Delta c_n = c_{n+1}- c_n &=& 3n^2+3n+1. \\ \nonumber
\Delta i_n = i_{n+1}- i_n &=& \frac{15}{2}n^2+ \frac{5}{2}n+1. \\ \nonumber
\Delta d_n = d_{n+1}- d_n &=& \frac{27}{2}n^2+\frac{9}{2}n+1. 
\end{eqnarray}

Since, $\Delta^2 y_n = y_{n+2}-2y_{n+1}+y_n$, we get the following
identities.

\begin{eqnarray} \label{secondeqns}
t_n - 2t_{n+1} + t_{n+2} &=& n+2. \\ \nonumber
o_n - 2o_{n+1} + o_{n+2} &=& 4n +4. \\ \nonumber
c_n - 2c_{n+1} + c_{n+2} &=& 6n+6. \\ \nonumber
i_n - 2i_{n+1} + i_{n+2} &=& 15n+10.  \\ \nonumber
d_n - 2d_{n+1} + d_{n+2} &=&  27n + 18.
\end{eqnarray}

Since $\Delta^3 y_n = y_{n+3}-3y_{n+2}+3y_{n+1}-y_n$, we get the
following identities.
\begin{eqnarray} \label{thirdeqns}
t_{n+3} - 3t_{n+2} + 3t_{n+1} - t_n &=& 1. \\ \nonumber
o_{n+3} - 3o_{n+2} + 3o_{n+1} - o_n &=& 4. \\ \nonumber
c_{n+3} -3c_{n+2} + 3c_{n+1} -c_n &=& 6. \\ \nonumber
i_{n+3} -3i_{n+2}+3i_{n+1}-i_n &=& 15. \\ \nonumber
d_{n+3} -3d_{n+2}+3d_{n+1}-d_n&=&27. 
\end{eqnarray}

The fourth forward difference is given by $\Delta^4
y_n=y_{n+4}-4y_{n+3}+6y_{n+2}-4y_{n+1}+y_n$. Consequently, we get the
following identities.
\begin{eqnarray} \label{fourtheqns}
t_{n+4}-4t_{n+3}+6t_{n+2}-4t_{n+1}+t_n &=& 0 \\ \nonumber 
o_{n+4}-4o_{n+3}+6o_{n+2}-4o_{n+1}+o_n &=& 0 \\ \nonumber 
c_{n+4}-4c_{n+3}+6c_{n+2}-4c_{n+1}+c_n &=& 0 \\ \nonumber 
i_{n+4}-4i_{n+3}+6i_{n+2}-4i_{n+1}+i_n &=& 0 \\ \nonumber 
d_{n+4}-4d_{n+3}+6d_{n+2}-4d_{n+1}+d_n &=& 0 
\end{eqnarray}

Consequently, it follows that all the  platonic numbers satisfy
the linear recurrence relation given by Equation \ref{lineareqns}.

From Equations \ref{secondeqns} and \ref{thirdeqns}, we get
$\Delta^2 t_n - 2 \Delta^3 t_n = n$. Consequently, 
\begin{equation} \label{bingoeqn}
3t_n-8t_{n+1}+7t_{n+2}-2t_{n+3} = n.
\end{equation}

Equation \ref{bingoeqn} implies that every integer can be written as
an integer combination of four tetrahedral numbers.

Again, from Equations \ref{secondeqns} and \ref{thirdeqns}, we
get $\Delta^2 o_n - \Delta^3 o_n = 4n$. Therefore,
\[
2 o_n - 5 o_{n + 1} + 4 o_{n + 2} - o_{n + 3} = 4n.
\]

Therefore, we conclude that integers divisible by $4$ can be written
as integer combinations of four octahedral numbers.

Similarly, since $\Delta^2 c_n - \Delta^3 c_n = 6n$, 
we get
\[
2 c_n - 5 c_{n + 1} + 4 c_{n + 2} - c_{n + 3} = 6n,
\]
which implies that the integers which are divisible by $6$ are integer
combinations of four cubes.

From Equations \ref{secondeqns} and \ref{thirdeqns}, we get  $3\Delta^2
i_n - 2\Delta^3 i_n = 45n$, therefore
\[
5i_n-12i_{n+1}+9i_{n+2}-2i_{n+3}=45n.
\]

Hence  integers that are divisible by $45$ can be written
as integer combinations of four icosahedral numbers. 

Finally, from Equations \ref{secondeqns} and \ref{thirdeqns}, we get
$3\Delta^2 d_n - 2\Delta^3 d_n = 54n$, therefore

\[
5d_n-12d_{n+1}+9d_{n+2}-2d_{n+3}=54n.
\]

Hence we conclude that integers that are divisible by $54$ are integer
combinations of four dodecahedral numbers.

Note that since the second forward differences are linear in
$n$, some  combination of these differences might  lead to a
proof  of the  Pollock's conjecture.  We leave  that  to the
reader to explore. In our study of the Pollock's conjecture,
we  observed that  we could  write the  numbers  between two
platonic numbers as sums of at most five Platonic numbers.

\begin{exm}{\em Integers between $o_5=85$ to $t_8=120$ written as sums of at most five different Platonic numbers.
\begin{adjustwidth}{-.85in}{.5in}
\[\begin{array}{lllllll}
86 = 85+1, & 87 = 85+1+1, & 88 = 85+1+1+1, & 89 = 85+4, \\ 90=85+4+1,
& 91 = 85+4+1+1, & 92 = 85+4+1+1+1, & 93 = 85+8, \\ 94=85+8+1,&
95=85+10, & 96=85+10+1, & 97=85+12, \\ 98=85+12+1, & 99 = 85+12+1+1, &
100=85+12+1+1+1, & 101=85+12+4,\\ 102=85+12+4+1, & 103=85+12+4+1+1, &
104=85+10+8+1, &105 = 85+20, \\ 106=85+20+1, & 107=85+20+1+1, &
108=85+20+1+1+1, & 109=85+20+4, \\ 110=85+20+4+1, & 111=85+20+4+1+1, &
112=85+27, &113=85+27+1, \\ 114=85+27+1+1, & 115=85+27+1+1+1, &
116=85+27+4, &117=85+27+4+1, \\ 118=85+27+4+1+1, & 119=85+20+12+1+1

\end{array} \]
\end{adjustwidth}
} \end{exm}

Thus, we add one more conjecture to the collection of the many beautiful
unproved conjectures on platonic numbers.

\begin{conjecture}
Every integer can be written as a sum of at most five different
Platonic numbers.

\end{conjecture}

\section{Modular Periods of Platonic numbers.}\label{details}

In this  section, we derive the modular  periods of platonic
numbers.  We use  the formulas  (\ref{formulas})  in Section
\ref{intro} for the proof of the following proposition.

\begin{prop} \label{periodprop}
Let $d>1$ be an integer. 

\begin{enumerate}

\item Consider the sequence of tetrahedral numbers $t_n$.  Let $p^t$ denote the 
period of $t_n$ mod $d$.
\begin{case}

\item $d$ is an even integer. 
\[p^t = \left \{\begin{array}{llllllllllll}
 6d & \mbox{if $d$ is divisible by $3$}; \\
2d & \mbox{otherwise.}   
\end{array}
\right.
\]

\item  $d$ is an odd integer. 
\[p^t = \left \{\begin{array}{llllllllllll}
 3d & \mbox{if $d$ is divisible by $3$}; \\
d & \mbox{otherwise.}   
\end{array}
\right.
\]

\end{case}

\item  Let $p^o$  denote of  the period  of the  sequence of
  octahedral numbers $o_n$ mod $d$.
\[p^o = \left \{\begin{array}{llllllllllll}
 3d & \mbox{if $d$ is divisible by $3$}; \\
d & \mbox{otherwise.}   
\end{array}
\right.
\]

\item If $p^c$  denotes the period of the  sequence of cubes
  mod $d$, then $p^c=d$.

\item Let $p^i$ denote the period of the sequence of icosahedral numbers mod $d$.
\[p^i = \left \{\begin{array}{llllllllllll}
 2d & \mbox{if $d$ is even}; \\
d & \mbox{otherwise.}   
\end{array}
\right.
\]

\item Let $p^d$ denote the sequence of dodecahedral numbers, then
\[p^d = \left \{\begin{array}{llllllllllll}
 2d & \mbox{if $d$ is even}; \\
d & \mbox{otherwise.}   
\end{array}
\right.
\]
\end{enumerate}

\end{prop}

\begin{proof}
\begin{enumerate}

\item We derive the periods of the sequence of tetrahedral numbers. 
\begin{case}
\item $d$ is even.

Let $d$ be not divisible by $3$. Since $d$ is even, we write
$d=2k$ for some $k$. For $n \geq 1$,
\begin{eqnarray*}
t_{n+4k}-t_{n} 
=\frac{(n+4k)((n+4k)+1)((n+4k)+2)}{6} - \frac{(n(n+1)(n+2)}{6} \\ \\
=\frac{2k}{3}\left (3n^2+12nk+6n+16k^2+12k+2 \right ).
\end{eqnarray*}
Since $d$  is not divisible by $3$,  and $t_{n+4k}-t_{n}$ is
an integer,  $3$ divides $(3n^2+12nk+6n+16k^2+12k+2)$. Hence
$t_{n+4k}-t_{n}=  2ks$ for  some integer  $s$. Consequently,
$t_{n+2d} \equiv t_{n}$ mod $d$, which implies $p^t =2d$.

On the  other hand, if $d$  is divisible by  $3$, then $d=6k$
for some integer $k \geq 1$. For any $n \geq 1$,
\begin{eqnarray*}
t_{n+36k}-t_n=6k(3n^2+108nk+6n+1296k^2+108k+2).
\end{eqnarray*}
Hence, $t_{n+36k} \equiv t_n$ mod $6k$. Consequently, $p^t=6d$.

\item $d$ is an odd integer.

Let $d$ be  not divisible by $3$. Since $d>1$ is odd, 
$d=2k+1$ for some integer $k \geq 1$. For $n \geq 1$,
\begin{eqnarray*} t_{n+(2k+1)}-t_n=\frac{(2k+1)}{6}\left(3n^2+6nk+9n+4k^2+10k+6\right).
\end{eqnarray*}
Now $d$ is not divisible by  $2$ or $3$. Hence $2k+1$ is not
divisible     by    $6$.    Consequently,     $6$    divides
$(3n^2+6nk+9n+4k^2+10k+6)$ because  $t_{n+(2k+1)}-t_n$ is an
integer.  Therefore  $t_{n+2k+1}-t_{n}=  (2k+1)s$  for  some
integer $s$. Thus, $p^t=d$.

Now consider  the case  when $d$ is  divisible by  $3$. Then
$d=3(2k+1)$, for some $k \geq 0$. For $n \geq 1$, 
\begin{eqnarray*}
t_{n+3d}-t_n = \frac{3(2k+1)}{2} ( 3n^2+54nk + 33n+324k^2+378k+110).
\end{eqnarray*}
Since $t_{n+3d}-t_n$ is  an integer, $2$ divides $(3n^2+54nk
+   33n+324k^2+378k+110)$.  Therefore  $t_{n+9(2k+1)}-t_{n}=
3(2k+1)s$  for some integer  $s$. Consequently,  $p^t=3d$ in
this case.

\end{case}
\item Next, we consider the sequence of octahedral numbers $o_n$.

Let $d$ be not divisible by $3$. For $n \geq 1$, 
\begin{eqnarray*}
o_{n+d}-o_d=\frac{d}{3}(6n^2+6nd+2d^2+1).
\end{eqnarray*}

Since $o_{n+d}-o_d$ is an  integer, and $d$ is not divisible
by   $3$,   we    get   $3$   divides   $(6n^2+6nd+2d^2+1)$.
Consequently, $o_{n+d} \equiv o_d$ mod $d$. Hence $p^o=d$ in
this case.

On the other hand, if $d$ is divisible by $3$, then $d=3k$ for some $k \geq 1$. For $n \geq 1$,
\begin{eqnarray*}
o_{n+3d}-o_n= 3k(6n^2+54nk+162k^2+1).
\end{eqnarray*}
Consequently, $o_{n+3d} \equiv o_d$ mod $d$. Hence $p^o=3d$.

\item $p^c=d$ because for $n \geq 1$,
\begin{eqnarray*}
c_{n+d}-c_n=(n+d)^3-n^3=3n^2d+3nd^2+d^3=d(3n^2+3nd+d^2).
\end{eqnarray*}

\item We now compute $p^i$.
Let $d$ be an odd integer. For $n \geq 1$,
\begin{eqnarray*}
i_{n+d}-i_n=\frac{d}{2}(15n^2+15nd-10n+5d^2-5d+2).
\end{eqnarray*}
Since  $d$ is  odd,  and $i_{n+d}-i_n$  is  an integer,  $2$
divides  $(15n^2+15nd-10n+5d^2-5d+2)$. Consequently, $p^i=d$
in this case.

On the other hand, when $d$ is even, $d=2k$ for some $k \geq 1$. For $n \geq 1$,
\begin{eqnarray*}
i_{n+2d}-i_n=\frac{(n+4k)(5(n+4k)^2-5(n+4k)+2)}{2} -\frac{n(5n^2-5n+2)}{2} \\ \\
= 2k(15n^2+60nk-10n+80k^2-20k+2).
\end{eqnarray*}
Thus $p^i=2d$.

\item Finally, we compute the periods of the sequence of dodecahedral numbers $d_n$.
Let $d$ be an odd integer. For $n \geq 1$,
\begin{eqnarray*}
d_{n+d}-d_n=\frac{(n+d)(9(n+d)^2-9(n+d)+2)}{2}
 -\frac{n(9n^2-9n+2)}{2}
\\ \\= \frac{d}{2}(27n^2+27nd-18n+9d^2-9d+2).
\end{eqnarray*}
Since $d$ is odd and $d_{n+d}-d_n$ is an integer, we get $2$
divides  $(27n^2+27nd-18n+9d^2-9d+2)$. Consequently,  we get
$p^d = d$ in this case.

On the other hand, when $d$ is even, $d=2k$ for some $k \geq 1$. For $n \geq 1$,
\begin{eqnarray*}
d_{n+2d}-d_n=
2k(27n^2+108nk-18n+144k^2-36k+2).
\end{eqnarray*}
Thus, $p^d=2d$.
\end{enumerate}
\end{proof}


\begin{thebibliography}{99}

\bibitem{cg} John H Conway and Richard K Guy, {\em The Book of
  Numbers}, Springer-Verlag, Copernicus, New York, 1996.

\bibitem{dickson} Dickson, L. E., {\em History of the Theory of
  Numbers, Vol. 2: Diophantine Analysis}, Dover,  New York, 2005.

\bibitem{engstorm}  Engstrom, H.  T.  {\em Periodicity  in
  sequences   defined   by   linear  recurrence   relations},
  Proceedings  of the  National Academy  of Sciences  of the
  United  States  of America,  16  (10) (1930),  663-665.  


\bibitem{kim} Hyun Kwang Kim, {\em On regular polytope numbers},
Proceedings of the American Mathematical Society, 131 (2003), 65-75.

\bibitem{pollock} Pollock, F. {\em On the Extension of the Principle
  of Fermat's Theorem of the Polygonal Numbers to the Higher Orders of
  Series Whose Ultimate Differences Are Constant. With a New Theorem
  Proposed, Applicable to All the Orders},  Abs. Papers
  Commun. Roy. Soc. London 5, 922-924, 1843-1850.


\end{thebibliography}
\end{document}